\documentclass{article}

\usepackage{amsmath,amsthm,amsfonts}
\usepackage{graphics}
\usepackage{epsfig} 
 \usepackage[colorlinks=true]{hyperref}

\usepackage{amssymb,bbm}

\usepackage[top=2cm, bottom=2cm, left=3cm, right=3cm]{geometry}

\newtheorem{theorem}{Theorem}[section]

\newtheorem{lemma}[theorem]{Lemma}

\newtheorem{remark}{Remark}

\newcommand{\ep}{\varepsilon}

\newcommand{\abs}[1]{\left|#1\right|}
\newcommand{\R}{{\mathbb {R}}}
\newcommand{\norm}[2]{\left\| #1 \right\|_{#2}}
\newcommand{\sshort}{S^{\text{short}}}
\newcommand{\slong}{S^{\text{long}}}

\begin{document}

\title{Global existence for the kinetic chemotaxis model without pointwise memory effects, and including internal variables
\footnote{The authors thank B. Perthame for fruitful discussions and challenging directions of research about these subjects. VC is grateful to University of Edinburgh for the kind hospitality during a one week visit.}
}

\date{December 15, 2007}

\maketitle

\centerline{\scshape Nikolaos Bournaveas }
\medskip
{\footnotesize
 \centerline{University of Edinburgh, School of Mathematics }
   \centerline{JCMB, King's Buildings, Edinburgh EH9 3JZ, UK }
   \centerline{\tt N.Bournaveas@ed.ac.uk}
}
\medskip

\centerline{\scshape Vincent Calvez }
\medskip
{\footnotesize
 \centerline{\'Ecole Normale Sup\'erieure, D\'epartement de Math\'ematiques et Applications }
   \centerline{45 rue d'Ulm,
F 75230, Paris, cedex 05, France}
\centerline{\tt Vincent.Calvez@ens.fr}
}
\bigskip

\begin{abstract}
This paper is concerned with the kinetic model of Othmer-Dunbar-Alt for bacterial motion. Following a previous work, we apply the dispersion and Strichartz estimates to prove global existence under several borderline growth assumptions on the turning kernel. In particular we study the kinetic model with internal variables taking into account the complex molecular network inside the cell.
\end{abstract}

\medskip

\noindent{Classification (AMS 2000) Primary: 92C17, 82C40; Secondary: 35Q80, 92B05}

\medskip

\section{Introduction and results}

In biology, several key processes of cellular spatial organisation are driven by chemotaxis. The effective mechanism by which individual cells undergo directed motion varies among organisms. We are particularly interested here in bacterial migration, characterized by the smallness of the cells, and their ability to swim up to several orders of magnitude in the attractant concentration. Several models, depending on the level of description, have been developed mathematically for the collective motion of cells \cite{Perthame04,PerthameBook}. Among them the kinetic model due to Othmer, Dunbar and Alt (ODA) \cite{Alt, ODA}, describes a population of bacteria in motion ({\em e.g. E. Coli} or {\em B. Subtilis}) \cite{ErbanOthmer04} in a field of chemoattractant (a process called {\em chemokinesis}). These small cells are not capable of measuring any head-to-tail gradient of the chemical concentration, and to choose directly some preferred direction of motion towards high concentrated regions. Therefore they develop an indirect strategy to select favourable areas,  by detecting a sort of time derivative in the concentration along their pathways, and to react accordingly \cite{MK72}. In fact they undergo a jump process where free movements (runs) are punctuated by reorientation phenomena (tumbles) \cite{WWS03}. For instance it is known that {\em E. Coli} increases the time spent in running along a favourable direction \cite{MK72,BB74,ErbanOthmer04}.

This jump process can be described by two different informations. First cells switch the rotation of their flagella, from counter-clockwise CCW (free runs) to clockwise CW (reorientation, or tumbling phase), and conversely. This decision is the result of a complex chain of reactions inside the cells, driven by the external concentration of the chemoattractant \cite{GarrityOrdal,SPO,WWS03}. 
Then cells select a new direction. Although we expect large organisms (like { the slime mold amoebae} {\em D. discoideum}) to choose directly a favourable direction, bacteria are unable to do so, and they randomly choose a new direction of motion. Actually some { directional persistence} may influence this selection, privileging some angles better than others. 
However we will not consider inertia here for simplicity.

{ From the molecular point of view, the frequency of tumbling events is driven by a regulatory protein network made of the membrane receptor complex (MCP), the switch complex located at the flagella motor, and six main proteins in between (namely CheA, CheW, CheY, CheZ, CheB and CheR -- more are involved in {\em B. Subtilis}, but the whole picture is similar \cite{GarrityOrdal}). This regulatory network exhibits a remarkable excitation/adaptation process \cite{BB74,SBB86,SPO}. When attractant concentration increases suddenly the tumbling frequency decreases in a short time scale (excitation), but increases back to the basal activity after a while (adaptation). This allows bacteria to follow favorable pathways over several orders of the concentration magnitude. Note that a similar adaptation process is involved in bigger organisms like {\em D. discoideum} \cite{Hofer,OthmerSchaap}. Realistic models have been proposed based on the complete regulatory network \cite{HauriRoss,SPO}, as well as toy models capturing the key behavior (basically made of a two species relaxing ODE system \cite{ErbanOthmer07}). Note that this network is also known to select positive perturbations of the chemoattractant concentration only \cite{BB74}, and to be highly sensitive to very low changes in the chemoattractant concentration \cite{SBB86}.}

As a drift-diffusion limit of the ODA kinetic model, one recovers the so-called Keller-Segel model \cite{HillenOthmer,CMPS,CDMOSS}, where diffusion and chemosensitivity coefficients can be derived from the mesoscopic description. The Keller-Segel model exhibits a remarkable dichotomy where cells aggregate if they are sufficient enough, and disperse if not \cite{Horstmann}. Particularly in the two dimensional case, the total mass of cells is the key parameter which { selects between these phenomena {(respectively global existence {\em versus} blow-up in finite time)}. This simple alternative is depicted in the whole space $\R^2$ in \cite{BDP}.} In the three dimensional case however, the relevant quantity ensuring global existence is rather the $L^{3/2}$ norm of the initial cell density \cite{CoPeZa}. Therefore it is of interest to ask the question of global existence at the mesoscopic level. As far as we know, no blow-up phenomenon has been found in the ODA kinetic model. 

{ The goals of this paper are the  following two.}
First we investigate global existence theory for several kinetic models depending on the growth of the reorientation kernel with respect to the chemical. In a previous work we succesfully applied dispersion and Strichartz estimates to kinetic models including delocalization effects \cite{BCGP}, that can be either a time delay effect due to intracellular dynamics, or  some measurement at the tip of a cell protrusion. Those techniques are applied here to a class of assumptions where the reorientation kernel is actually independent of the (inner and outer) velocities. Those assumptions are very rough from the biological point of view, but they aim to determine the critical growth of the turning kernel ensuring global existence.
On the other hand we apply those ideas to a more realistic kinetic model including internal molecular variables, improving the results of \cite{ErbanHwang}. We present general assumptions for global existence that can be satisfied by the two species excitative/adptative ODE system, or more generally by a complex network.

We consider the following ODA kinetic model for bacterial chemotaxis: 
\begin{subequations}\label{kinmodel}
\begin{align}
\partial_{t} f + v\cdot\nabla_{x} f &= \int_{v'\in V} T[S](t,x,v,v') f(t,x,v') dv' \nonumber \\
&\qquad - \int_{v'\in V} T[S](t,x,v',v) f(t,x,v) dv' \ ,\quad t>0\ , x\in \R^d  \label{eq:ODA f}\\
- \Delta S + S &= \rho(t,x)=\int_{v\in V} f(t,x,v) dv\ , 
\end{align}
\end{subequations}
associated with the initial condition $f(0,x,v)= f_0(x,v)$. The space density of cells is denoted by $\rho(t,x)$. We assume in this paper that the space dimension is $d=2$ or $d=3$. We assume as usual that the set $V\in \R^d$ of { admissible} cell velocities is bounded.
{The free transport operator $\partial_t f + v\cdot\nabla_x f$ describes the free runs of the bacteria which have velocity $v$.} On the other hand, the scattering operator in the right hand side of \eqref{eq:ODA f} expresses the reorientation process (tumbling) occuring during the bacterial migration towards regions of high concentration in chemoattractant $S$.

\subsection*{Partial review of plausible reorientation mechanisms.}

We review below the assumptions existing in the literature concerning the reorientation kernel, in order to motivate the forthcoming work.

\subsubsection*{Delocalization effects.}

In a previous article \cite{BCGP}, we considered mild assumptions of the type:
\begin{equation}
\label{eq:turning kernel delay} 
0\leq T[S](t,x,v,v') \leq C \Big( 1 + S(t,x-\ep v') + |\nabla S|(t,x-\ep v') \Big) , \end{equation}
or,
\begin{equation}
\label{eq:turning kernel protrusion}
0\leq T[S](t,x,v,v') \leq C \Big( 1 + S(t,x+\ep v) + |\nabla S|(t,x+\ep v) + |D^2S|(t,x+\ep v) \Big). \end{equation}
Those assumptions were studied for example in  \cite{CMPS,HKS} { in two or three dimensions of space.}

Under assumption \eqref{eq:turning kernel delay}, the bacteria take the decision to reorient with the probability $\lambda[S] = \int T[S](t,x,v',v)\ dv'$, and then {choose} a new direction randomly. Therefore the turning frequency increases once cells have entered a favourable area, say (where some delay effect { due to internal dynamics} is expressed by the space shifting $ - \ep v'$; the concentration measurement is performed at position $x- \ep v'$ by the cell with velocity $v'$, turning at position $x$).
Intuitively, the cells increase the turning frequency to be confined in highly concentrated areas. 

The hypothesis \eqref{eq:turning kernel protrusion} is even more intuitive: cells, when they decide to turn (due to a complex averaging over the surrounding area within a radius $\simeq\ep$), simply choose a better new direction $v$ with higher probability. This anticipation measurement can be the result of sending protrusions { in the surrounding}, or considering that the cells have some finite radius with receptors located all over the membrane (see also \cite{HPS} for a similar interpretation at the parabolic level -- volume effects have also been considered at the kinetic level in \cite{ChalubRodrigues}). However this interpretation is hardly relevant for bacteria which are small cells, unable to feel gradients and to send protrusions. 

\begin{remark}
The gradient in assumption \eqref{eq:turning kernel delay} has to be motivated because we highlight that bacteria cannot feel gradients of chemical concentration. As a matter of fact, from a homogeneity viewpoint, $\nabla S$ has the same weight as the time derivative $\partial_t S$. Therefore we can replace indeed \eqref{eq:turning kernel delay} by the assumption
\[ 0\leq T[S](t,x,v,v') \leq C \Big( 1 + S(t,x-\ep v') + |\partial_t S|(t,x-\ep v') \Big) \ , \]
which makes sense biologically (although a more realistic assumption is expressed for example in \cite{DolakSchmeiser}, see below \eqref{eq:directional derivative}).
To see that $\nabla S$ and $\partial_t S$ do have the same homogeneity, observe that
\[ \partial_t S  = G* \partial_t\rho = - G*\nabla\cdot j \ ,   \]
where $G$ is the Bessel potential, and the flux $j$ is given by 
\[ j(t,x) = \int_V v f(t,x,v)\ dv \  , \quad |j(t,x) | \leq (\max_{v\in V} |v|) \rho(t,x)\ .  \]
As a consequence, { in the three dimensional case we have} 
\[ |\partial_t S| \simeq \left| \frac1{|x|}*(\nabla \cdot j) \right|\simeq \frac1{|x|^2}* |j| \lesssim \frac1{|x|^2}* \rho \simeq |\nabla S|\ . \]
\end{remark}

The dispersion lemma turned out to be  a powerful tool for dealing with  those assumptions (even the second derivative of $S$ can be 
added in \eqref{eq:turning kernel protrusion}). 
It turns out that putting together  those two hypotheses \eqref{eq:turning kernel delay} and \eqref{eq:turning kernel protrusion} is a much harder task (due to the fact that we loose the benefit of the decay term in the balance law { along the estimates}). 
{ Some progress} in this direction was recently made in \cite{CMPS, HKS, BCGP} but the whole picture is not clear so far. For example in \cite{BCGP} it was shown that in $d=3$  dimensions we have global existence of weak solutions if
\begin{equation}\label{1}
0 \leq T[S](t,x,v,v') \lesssim  1 +  S(t,x+\ep v) + S(t,x-\ep v') + \abs{\nabla S(t,x+\ep v)} + \abs{\nabla S(t,x-\ep v')}, 
\end{equation}
provided that the initial data are small in the critical space $L^{3/2}$. If \eqref{1} is strengthened by dropping the last term,
then a global existence result was established without a smallness assumption on the initial data.  The proofs use the 
dispersion and Strichartz estimates of \cite{CP} and rely on the delocalization effects induced by $x+\ep v$ and $x-\ep v'$.

Interestingly the fact that some directed motion emerges from turning kernels which resemble to assumptions \eqref{eq:turning kernel delay}, \eqref{eq:turning kernel protrusion} or \eqref{1} -- as pointed out by the diffusion limit -- seems to involve a completely different  
mechanism from the following commonly described behaviour in {\em E. coli}.

\subsubsection*{Persistence of motion in the good directions.}
As opposed to the previous set of hypothesis, it is commonly accepted  that bacteria increase the time spent in running in a favourable direction \cite{WWS03,ErbanOthmer04}. That is, the turning kernel is expected to decrease as the chemical concentration increases along the cell's trajectory, like  \begin{equation} T[S](v,v') = T_0 + \psi(S_t + v'\cdot \nabla S)\ , \label{eq:directional derivative}\end{equation} where $\psi$ is nonnegative and decreasing, and $S_t + v'\cdot \nabla S$ denotes the directional derivative along the free run before turning (see \cite{ErbanHwang}, \cite{DolakSchmeiser} { where this hypothesis is injected in a model for {\em D. discoideum} self-organization, and its drift-diffusion limit is derived}). One may think of $\psi$ to be: $\psi(\eta) = 0$ if $\eta>0$ and $\psi(\eta) = 1$ if $\eta<0$ for instance. { Actually, in \cite{ErbanHwang} the authors explicit two behaviour caricatures, where cells might "perfectly avoid going in wrong directions", or "perfectly follow good directions". The latter is stressed out there and leads the system to regular solutions, intuitively, whereas the former might develop singularities where cells aggregate.}

{ The above mechanism is also part of more complex models including internal variables (which is reviewed and analysed further below). In fact some molecular concentration denoted by $y$ (standing for the phosphorylated CheY-P) which induces a tumbling behavior, is actually reduced under attractant binding to the membrane receptor (excitation phase). The chemical chain of reactions is in fact inhibited under activation of the membrane complex receptor. On the contrary, expression of a repellent activates this internal network, favouring tumbling.} Global existence theory for such a class of models has been discussed in \cite{ErbanHwang} for the one dimensional case.

\subsubsection*{Internal dynamics}
Complex models of bacterial motility include a cascade of chemical reactions. This chain of activator/inhibitor reactions links the evaluation of the chemical concentration by the membrane receptors to the rotational switch of the flagella, inducing { or inhibiting} the tumbling phase. Several works propose a chemical network describing this complexity \cite{HauriRoss,SPO}. In particular, the global short term excitation/mid term adaptation is crucial for the cells to crawl up across levels of magnitude of the chemical concentration. Caricatures of such an excitation/adaptation process are depicted in \cite{ErbanOthmer04,DolakSchmeiser} for instance. However we will keep in this paper  the necessary abstract level required for our purpose ({ for an illustrative example, see section \ref{sec:internal}}).

In the following, $y\in \R^m$ denotes the whole internal state of the cells, which can correspond to huge data of { molecular} concentrations in the chemical network (in fact $m=2$ in the caricatural excitating/adaptating system).  
In accordance with previous notations, $p(t,x,v,y)$ denotes the cell density at position $x$, velocity $v$, and with internal state $y$. As before, $f(t,x,v) = \int_y p(t,x,v,y)\ dy$ is the cell density in position$\times$velocity space. On the other hand we introduce $\mu(t,x,y) = \int_v p(t,x,v,y) \ dv$, and as usual $\rho(t,x) = \int_{v,y} p(t,x,v,y) \ dvdy$.
The chemical potential is given by a mean-field equation $-\Delta S + S = \rho(t,x)$. But this could be extended to a more realistic influence of the internal state on the chemical secretion (as it is in \cite{DolakSchmeiser})
\[- \Delta S + S = \int_y \omega(y) \mu(t,x,y)\ dy \ ,\] under suitable assumptions on the weight $\omega$.
The dynamic inside an individual cell is driven by an ODE system representing the protein network in an abstract way:
\[ \frac{dy}{dt} = G\big(y,S(t,x)\big)\ , \quad y\in \R^m\ . \]
The cell master equation describing the run and tumble processes, and the chemical potential equation are respectively:
\begin{subequations}\label{kinmodel_y intro}
\begin{align}
\partial_{t} p + v\cdot\nabla_{x} p + \nabla_y \cdot \Big( G(y,S) p \Big) &=   \int_{ v'\in V} T(t,x,v,v',y) p(t,x,v',y) dv' \nonumber\\
   &\qquad  -  \int_{ v'\in V} T(t,x,v',v,y) p(t,x,v,y) dv' \  ,\\
 -\Delta S + S & = \rho \ ,   
\end{align}
\end{subequations}

The turning kernel $T$ can be decomposed in this context as product between a turning frequency $\lambda[y]$, depending on the internal state only, and a reorientation $K(v,v')$ which may describe some persistence in the choice of a new direction with respect to the old one. Without loss of generality here we assume that $K(v,v')$ is constant and renormalized as being $K(v,v') = 1/|V|$.

It is worth noticing that this realistic kinetic model may contain enormous informations on the microscopic cell biology, and links different scales of description, because we eventually end up with a cell population $\rho(t,x)$. 

\medskip

As a partial conclusion, we observe that several scenarios with different underlying kinds of hypotheses, drive the system to positive chemotaxis (at least considering the formal drift-diffusion limit of those).

\subsection*{Statement of the main results}

In this paper we investigate the critical growth of the turning kernel in terms of space norms of the chemical which ensures the global existence for the kinetic model. { In particular 
we consider control on the turning kernel without any dependence upon the velocity variables, that is with some abuse of notations:
\[ 0\leq T[S](t,x,v,v') \leq T[S](t,x)\ , \] under suitable conditions on the growth of $T[S]$.}
We exhibit examples in 2D and 3D, restricting ourselves to some $L^p$ norms of the chemical (and not of its gradient for instance) for which our method appears to be borderline.
In particular, it is  natural to ask (see Section 3 of \cite{CMPS} and the concluding remarks in \cite{BCGP}) 
whether global existence can  be established under a hypothesis of the form 
\begin{equation}\label{hypothesisA}
 0 \leq T[S](t,x,v,v') \leq C\Big( 1+ \norm{S(t,\cdot)}{L^{\infty}(\R^d)}^{\alpha}\Big)\ ,
\end{equation}
 where $\alpha>0$. 

\subsubsection*{Exponential growth in dimension 2}

Consider first the case of dimension $d=2$. It is easy to see using the methods of \cite{CMPS} that we have global existence
for any exponent $\alpha>0$ within \eqref{hypothesisA}.  In analogy with global existence results for nonlinear wave or Schr\"odinger equations 
\cite{IMM1, IMM2, NO1, NO2} we can ask whether the turning kernel can grow exponentially: 
\begin{equation}\label{hypothesisB}
 0 \leq T[S](t,x,v,v') \leq C\left( 1+ \exp\left[\norm{S(t,\cdot)}{L^{\infty}(\R^2)}^{\beta}\right] \right) .
\end{equation}
We will show that this is actually possible: if $0<\beta<1$ then we have global existence for large data; if $\beta=1$ 
we have global existence for initial data of small mass. Our proof requires $M<\pi$, but we don't know if this bound is optimal. 
Also, we  don't know if we { may} have blow-up for large $M$ or for exponents $\beta>1$.

 We shall prove the following 
\begin{theorem} \label{exp3C} 
Consider the system \eqref{kinmodel} in $d=2$ dimensions under hypothesis \eqref{hypothesisB} and let $1<p<2$.
Assume $0 <\beta \leq 1$.  
If $\beta =1 $ assume also that  $M<\pi$, where $M=\int_{V} f_0(x,v) dv$ is the { total mass of cells.}  
Then if $f_0\in L^{1}_{x}L^{p}_{v} \cap L^{1}_{x,v}$ then 
\eqref{kinmodel} 
has a  global weak solution $f$ with $f(t)\in L^{p}_{x}L^{1}_{v} \cap L^{1}_{x,v}$.

\end{theorem}

\subsubsection*{Almost $L^\infty$ growth in 3D}

Naturally, from the global existence point of view we address the question of a $\|S\|_\infty$ growth of the turning kernel in the case of $d=3$ dimensions: $T[S]\leq C\big( 1+ \|S\|_\infty\big)$. Actually it cannot be handled using our method in three dimensions so far. Even in the simpler case $T[S] \leq C\big(1+ S(t,x)\big)$ our dispersion method fails.   
Puzzling enough, if $T[S]= C\big(1 + S(t,x)\big)$ then a very simple symmetrization trick does perfectly the job (see section \ref{sec:dispersion}). 

It was noticed
in \cite{CMPS} that if $\alpha<1$ in \eqref{hypothesisA} then we have global existence (a sketch of the proof will be given 
in Section \ref{sublinear}). The case $\alpha=1$ remains open.
In this direction we will use  the methods of \cite{BCGP} to show that we have global existence under the assumption
\begin{equation}\label{lra2}
 0\leq  T[S](t,x,v,v') \leq C\left( 1+ \norm{S(t,\cdot)}{L^{r}(\R^3)}^{\alpha}\right)\ ,
\end{equation}
where  $0<\alpha < \frac{r}{r-3}$ and  $r$ can be arbitrarily large. Notice that   $\frac{r}{r-3} \to 1^{+}$
as $r\to\infty$, { which is coherent with the above obstruction}. More precisely we shall prove the following 

\begin{theorem} \label{lra1}
Let $d= 3$ and $1<p<3/2$. Suppose that that the turning kernel $T[S]$ satisfies hypothesis \eqref{lra2} for some $r$ and $\alpha$  { verifying:}   
if $1\leq r \leq 3$,  $\alpha$ can be any positive number, whereas in case of $3< r <\infty$,  
$0<\alpha < \frac{r}{r-3}$. 
Then if $f_0\in L^{1}_{x} L^{p}_{v}\cap L^{1}_{x,v}$ then the kinetic model \eqref{kinmodel}
has a global weak solution with $f(t)\in L^{p}_{x}L^{1}_{v}\cap L^{1}_{x,v}$.
\end{theorem}

If we assume that the turning kernel satisfies \eqref{hypothesisA} with $\alpha =1$,
we can use the Strichartz estimates of \cite{CP} to show global existence, provided that the critical norm 
$\norm{f_0}{L^{3/2}\left(\R^{6}_{x,v}\right)}$ is small.

\begin{theorem} \label{3dsmall}
Let $d=3$ and assume that the turning kernel satisfies
\[ 0 \leq T[S](t,x,v,v') \leq C\left( 1+ \norm{S(t,\cdot)}{L^{\infty}(\R^3)}\right)\ .\]
 Assume also that 
$f_0 \in L^{1}\left(\R^{6}_{x,v}\right) \cap
L^{3/2}\left(\R^{6}_{x,v}\right)$
and that $\norm{f_0}{L^{3/2}\left(\R^{6}_{x,v}\right)}$ is sufficiently
small. Then \eqref{kinmodel} has a global weak solution. 
\end{theorem}

\subsubsection*{Internal dynamics.}


We shall prove the following theorem for global existence in three dimensions of space.
\begin{theorem} 
\label{the:internal}
Let $d=3$.
Assume that the turning kernel has the form $T = \lambda[y]\times K(v,v')$ where $K$ is uniformly bounded, and $\lambda$ grows at most linearly: $\lambda[y]\leq C\big( 1+ |y|\big)$. On the other hand, assume that $G$ has a (sub)critical growth with respect to $y$ and $S$: there exists $0\leq \alpha<1$ such that
\[ |G|(y,S)\leq C\Big( 1 + |y| + S^\alpha\Big)\ . \]
Then there exists an exponent $1<p<3/2$ such that the system \eqref{kinmodel_y intro} admits globally existing solutions with $p\in L^p_xL^1_vL^1_y$.
\end{theorem}

\section{The dispersion lemma applied to kinetic chemotaxis, and the symmetrization trick}

\label{sec:dispersion}

In this section we present a direct application of the dispersion lemma \cite{CP} to system  \eqref{kinmodel}. As a consequence we are led to the following question, which is decoupled from \eqref{kinmodel}:
\begin{quote}
\emph{Investigate the critical norm for the turning kernel ensuring the bound
\[ 0 \leq T[S] \leq C\Big( 1 + \| \rho(t) \|_{L^{p}}\Big)\ , \]
for $p<d'$ in dimension $d$.}
\end{quote}
The rest of this paper will be devoted to this question of critical growth.

\begin{lemma} \label{lem:dispersion} Assume the turning kernel can be controlled without any dependence on the velocity variables $v$ nor $v'$:
\[ 0 \leq T[S](t,x,v,v') \leq T[S](t,x)
 .\] Then, applying the dispersion estimate, we get the following for $p\in [1,d')$:
 \[ \| \rho(t) \|_{L^p} \leq   \|f_0(x-tv,v)\|_{L^{p}_{x}L^{1}_{v}} + |V|^{1/p} \int_{s=0}^t (t-s)^{-\lambda} \int_x T[S](s,x)   \rho(s,x) \ dxds \ , \]
where $\lambda = d / p'$.
\end{lemma}

Observe that the condition $d<p'$ is crucially required here to ensure { further} the time integrability of the right-hand-side.

\begin{proof}
As usual we represent the solution of \eqref{kinmodel} as
\[f(t,x,v)\leq f_0(x-tv,v) + \int_{0}^{t} T[S](s,x-(t-s)v) \rho(s,x-(t-s)v) ds \ .\]
Using dispersion we get immediately 

\[\|f(t,x,v)\|_{L^{p}_{x}L^{1}_{v}} \leq \|f_0(x-tv,v)\|_{L^{p}_{x}L^{1}_{v}} + 
\int_{0}^{t} \frac{1}{(t-s)^{d(1-1/p)}} \Big\| T[S](s,x) \rho(s,x) \Big\|_{L^1_x L^p_v} \!  ds \ .\]


\end{proof}

As an observation, we state also a second lemma, which is interesting in its own right, but which will not be used in the sequel. Following \cite{Perthame04b}, it claims that a kernel which is symmetric with respect to $v$ and $v'$ ensures global existence. It is relevant from the mathematical point of view because we consider bounds that do not depend on $v$ and $v'$. It is biologically irrelevant however in the case of a purely symmetric kernel because no directed motion emerges in the drift-diffusion limit \cite{CMPS}.

\begin{lemma}
Consider the scattering equation,
\begin{equation}\label{13}
\partial_{t} f + v\cdot \nabla_{x} f = \int_{V} \left(K(t,x,v,v')f(t,x,v') - K(t,x,v',v)f(t,x,v)\right) dv'\ .
\end{equation}
and assume that $K$ is symmetric w.r.t. $v$ and $v'$, i.e.
\begin{equation}
K(t,x,v,v')= K(t,x,v',v) \geq 0 .
\end{equation}
Then all $L^p_xL^p_v-$norms of the density $f$ ($1\leq p<\infty$) are uniformly estimated like
$\norm{f(t)}{L^{p}_{x,v}} \leq \norm{f_0}{L^{p}_{x,v}}$. 
\end{lemma}

\begin{proof}
First rewrite \eqref{13} using the symmetry property. It becomes
\begin{equation}\label{13b}
\partial_{t} f + v\cdot \nabla_{x} f = \int_{V} K(t,x,v,v') \left(f(t,x,v') - f(t,x,v)\right) dv'\ .
\end{equation}
Next multiply \eqref{13b} by $p f^{p-1}(t,x,v)$ to get
\begin{align*}
& \partial_{t} f^{p} + v\cdot \nabla_{x} f^{p} = p \int f^{p-1}(t,x,v) 
K(t,x,v,v') \left(f(t,x,v') - f(t,x,v)\right) dv' 
\end{align*}
Integrate with respect to $x$ and $v$ to get
\begin{align*}
&\dfrac d{dt} \iint f^{p} dv dx   = 
 p \iiint f^{p-1}(t,x,v) 
K(t,x,v,v')\left(f(t,x,v') - f(t,x,v)\right) dv' dv dx\ .
\end{align*}
We can symmetrize the latter expression to obtain eventually
\begin{multline*}
 \dfrac d{dt} \iint f^{p} dv dx   = \\ 
 - \frac p2 \iiint K(v,v') 
 \left(f^{p-1}(t,x,v)-f^{p-1}(t,x,v')\right)\left(f(t,x,v) - f(t,x,v')\right) dv' dv dx \ .
\end{multline*}
Since $f\geq 0$ we have 
\[\left(f^{p-1}(t,x,v)-f^{p-1}(t,x,v')\right)\left(f(t,x,v) - f(t,x,v')\right) \geq 0\ , \]
because these two factors always have the same sign.
It follows that 
\begin{equation*}
\dfrac d{dt}  \iint f^{p} dv dx \leq 0  \ .
\end{equation*}
\end{proof}

\section{Exponential growth in $L^\infty$ in dimension 2}\label{d=2}

In this Section we prove Theorem \ref{exp3C}.  
Working as in the proof of Trudinger's inequality we expand the exponential into a power series
and use Young's inequality as in \cite{CMPS, BCGP} to estimate each term. The dispersion method is then used as in \cite{BCGP} through \eqref{lem:dispersion}.
Throughout these processes we keep track of the growth of the various constants in order to make sure that the resulting series converges.
A similar approach has been used in \cite{IMM1, IMM2, NO1, NO2} to study nonlinear wave and Schr\"odinger equations.
We will need the following two Lemmas.

\begin{lemma}\label{exp3A} Let 
$G(x)= \frac{1}{4\pi}\int_{0}^{\infty} e^{- \pi \frac{|x|^2}{s}} e^{-\frac{s}{4\pi}} \frac{ds}{s}$. 
There exists a  positive constant $A$ such that
\begin{equation}\label{exp3B}
G(x) \leq A + \frac{1}{2\pi} \abs{\log|x|\,}\ \ ,\ \ |x|\leq 1 .
\end{equation}
\end{lemma}

\begin{proof}
Fix $x$ with $|x|\leq 1 $. Write $G(x)=G_{1}(x) + G_{2}(x) + G_{3}(x)$ where
\begin{align*}
G_{1}(x)&=\frac{1}{4\pi}\int_{0}^{|x|^2} e^{- \pi \frac{|x|^2}{s}} e^{-\frac{s}{4\pi}} \frac{ds}{s} ,\\
G_{2}(x)&=\frac{1}{4\pi}\int_{|x|^2}^{1} e^{- \pi \frac{|x|^2}{s}} e^{-\frac{s}{4\pi}} \frac{ds}{s} ,\\
G_{3}(x)&=\frac{1}{4\pi}\int_{1}^{\infty} e^{- \pi \frac{|x|^2}{s}} e^{-\frac{s}{4\pi}} \frac{ds}{s} .
\end{align*}
For $G_{1}$ use $e^{-s/4\pi} \leq 1$ and then change variables $s\mapsto t$, where $s=|x|^2 t$, to get 
$G_{1}(x)\leq \frac{1}{4\pi}\int_{0}^{1} e^{-\pi\frac{1}{t}} \frac{dt}{t}=: A_{1}$.
For $G_{2}$ we have 
$G_{2}(x) \leq \frac{1}{4\pi}\int_{|x|^2}^{1} \frac{ds}{s} = \frac{-\log |x|}{2\pi}$.
 For $G_{3}$ use $e^{-\pi |x|^2 /s} \leq 1$ to get 
$G_{3}(x) \leq \frac{1}{4\pi}\int_{1}^{\infty} e^{-\frac{s}{4\pi}} ds =: A_{2}$.

\end{proof}

\begin{remark}
 In fact  the exact asymptotics of $G$ near the origin is: 
 \[ G(x) = -\frac1{2\pi} \log|x| + \gamma + \frac1{2\pi} \log 2 + o(1) \ ,  \] 
  where $\gamma$ is the Euler constant.
\end{remark}

\begin{lemma}\label{exp6A}
For $x > 0$ define $\Gamma(x)=\int_{0}^{\infty} t^{x-1} e^{-t} dt .$
Then (Stirling's formula)
\begin{align}
&n! = \Gamma(n+1) \sim \sqrt{2\pi n} \left(\frac{n}{e}\right)^{n} \ \ (n\to+\infty)   ,  \label{exp6B}\\
&\Gamma(x+1) \sim \sqrt{2\pi x} \left(\frac{x}{e}\right)^{x} \ \ (x\to+\infty)\label{exp6C} .
\end{align}
Moreover, for all $\beta>0$, $x>0$, 
\begin{equation}\label{exp6D}
 e^{x} > \frac{x^{\beta}}{\Gamma(\beta+1)} .
\end{equation}
\end{lemma}

\begin{proof}
\eqref{exp6B} and \eqref{exp6C} are well known. 
For \eqref{exp6D} we have 
\[\Gamma(\beta+1) = \int_{0}^{\infty} t^{\beta} e^{-t} dt > \int_{x}^{\infty}t^{\beta} e^{-t} dt >
x^{\beta} \int_{x}^{\infty} e^{-t} dt = x^{\beta} e^{-x} .\]

\end{proof}

\begin{proof}[Proof of Theorem \ref{exp3C}]
Recall from section \ref{sec:dispersion} that a control of the turning kernel like $ T[S] \leq C\big( \| \rho(t) \|_{L^{p}}\big)$, is sufficient to guarantee global existence. The rest of this section is devoted to the proof of this estimate.

Pick $1<p<2$ and set $\mu = p'>2$. In case of $\beta = 1$, assume in addition that $\mu < \frac{2\pi}M$.

Write $S=S^{long}+S^{short}$ where \[S^{long}= \left( \mathbbm{1}_{|x|> 1} G(x)\right) \ast \rho\ , \mbox{and}\quad   S^{short}=\left( \mathbbm{1}_{|x|\leq 1} G(x)\right) \ast \rho\ .\] 
Since $0<\beta \leq 1$ we have
\[\norm{S}{L^\infty}^{\beta}\leq \left(\norm{\slong}{L^{\infty}} + \norm{\sshort}{L^{\infty}}\right)^{\beta} \leq
\norm{\slong}{L^{\infty}}^{\beta} + \norm{\sshort}{L^{\infty}}^{\beta}\ ,\]
where we have used the fact that $(x+y)^{\beta}\leq x^{\beta} + y^{\beta}$ for $x,y>0$ and $0<\beta\leq 1$.
Therefore
\[
\exp\left\{\norm{S(t,\cdot)}{L^{\infty}}^{\beta}\right\} \leq 
\exp\left\{\norm{S^{long}(t,\cdot)}{L^{\infty}}^{\beta}\right\} \cdot  
\exp\left\{\norm{S^{short}(t,\cdot)}{L^{\infty}}^{\beta}\right\} .
\]
For $S^{long}$ we have
\[\norm{S^{long}}{L^\infty}^{\beta} \leq \norm{ \mathbbm{1}_{|x|> 1} G(x)}{L^{\infty}}^{\beta}
 \norm{\rho}{L^1}^{\beta} \leq c M^{\beta} ,\]
 where $c$ is a positive constant (depending on $\beta$), therefore
\begin{align*}
\exp\left\{\norm{S(t,\cdot)}{L^{\infty}}^{\beta}\right\} &  \leq 
 e^{c M^{\beta}}   \exp\left\{\norm{S^{short}(t,\cdot)}{L^{\infty}}^{\beta}\right\}\\
& = e^{c M^{\beta}} \left( 1 + \sum_{j=1}^{\infty} \frac{1}{j!} \norm{S^{short}(t,\cdot)}{L^\infty}^{j\beta} \right) .
\end{align*}
For $S^{short}$ we have
\[
\norm{S^{short}(t,\cdot)}{L^{\infty}} \leq \norm{G(x) \mathbbm{1}_{|x|\leq 1} }{L^{\mu j}} 
\norm{\rho}{L^{\frac{\mu j}{\mu j-1}}}
\]
For all $j\geq 1$ we have 
$\frac{\mu j}{\mu j-1} \leq \frac{\mu }{\mu -1}  = p$ therefore
\begin{align*}
\norm{S^{short}(t,\cdot)}{L^\infty} & \leq   
\norm{G(x) \mathbbm{1}_{|x|\leq 1}  }{L^{\mu j}} 
   \norm{\rho(t,\cdot)}{L^1}^{1-\frac1j}  \norm{\rho(t,\cdot)}{L^p}^{\frac1j}\\
&= \norm{G(x) \mathbbm{1}_{|x|\leq 1}  }{L^{\mu j}} 
   M^{1-\frac1j}  \norm{\rho(t,\cdot)}{L^p}^{\frac1j} ,
\end{align*}
therefore
\[\norm{S^{short}}{L^\infty}^{j\beta} \leq \norm{ G(x) \mathbbm{1}_{|x|\leq 1}  }{L^{\mu j}}^{j\beta} 
M^{j\beta-\beta}  
\norm{\rho(t,\cdot)}{L^p}^{\beta} .\]
Consequently
\begin{equation}
\exp\left\{\norm{S(t,\cdot)}{L^{\infty}}^{\beta}\right\} \leq 
e^{c M^{\beta}} \left( 1 +  \left[\sum_{j=1}^{\infty} \frac{1}{j!}
\norm{G(x) \mathbbm{1}_{|x|\leq 1} }{L^{\mu j}}^{j\beta} 
M^{j\beta}\right] M^{-\beta} \norm{\rho}{L^p}^{\beta}\right)\ . \label{exp3J}
\end{equation}
We need to guarantee that the series in the above right-hand-side converges. Using \eqref{exp3B} we have:
\begin{align*}
\norm{G(x) \mathbbm{1}_{|x|\leq 1} }{L^{\mu j}} & \leq 
\norm{A + \frac{1}{2\pi} \abs{\log|x|}\,  }{L^{\mu j} (|x|\leq 1)}\\
&\leq A \pi^{\frac{1}{\mu j}} + \frac{1}{2\pi} \norm{\log|x|\,}{L^{\mu j}(|x|\leq 1)}\ ,
\end{align*}
and also
\begin{align*}
 \norm{\log|x|\,}{L^{\mu j}(|x|\leq 1)} 
&=\left( 2\pi \int_{0}^{1} \left(\, - \log r \, \right)^{\mu j} r dr   \right)^{1/\mu j}\\
&\leq \left( 2 \pi \int_{0}^{\infty} s^{\mu j} e^{-2s} ds   \right)^{1/\mu j} \\
&\leq  \left( 2 \pi \int_{0}^{\infty} \frac{s^{\mu j}}{\frac{s^{\mu j}}{\Gamma(\mu j+1)}   } e^{-s} ds   \right)^{1/\mu j} 
\ \ \ \text{by}\ \ \ 
\eqref{exp6D}\\
&=\left( 2 \pi\right)^{\frac{1}{\mu j}} \left( \Gamma(\mu j+1)  \right)^{\frac{1}{\mu j}} \ .
\end{align*}
As a consequence
\begin{equation}\label{exp3H}
\norm{G(x) \mathbbm{1}_{|x|\leq 1} }{L^{\mu j}}  \leq A \pi^{\frac{1}{\mu j}} + \frac{1}{2\pi} \left( 2 \pi \right)^{\frac{1}{\mu j}} 
\left( \Gamma(\mu j+1)  \right)^{\frac{1}{\mu j}} .
\end{equation}
Then the infinite sum in \eqref{exp3J} can be estimated by 
\begin{equation}\label{exp3K}
\sum_{j=1}^{\infty} \frac{1}{j!} \left(A \pi^{\frac{1}{\mu j}} + \frac{1}{2\pi} \left( 2 \pi \right)^{\frac{1}{\mu j}} 
\left( \Gamma(\mu j +1) \right)^{\frac{1}{\mu j}}
\right)^{j\beta} M^{j\beta} .
\end{equation}
We'll show that for $\beta<1$ the series converges for any mass $M$, and that
for $\beta=1$ it converges thanks to the restriction $\frac{M\mu}{ 2\pi}<1$. Using the root test we have 
\begin{align*}
&\left( \frac{1}{j!} \left(A \pi^{\frac{1}{\mu j}} + \frac{1}{2\pi} \left( 2 \pi\right)^{\frac{1}{\mu j}} 
\left( \Gamma(\mu j +1) \right)^{\frac{1}{\mu j}}
\right)^{j\beta} M^{j\beta} \right)^{\frac{1}{j}} \\
& \ \ \ \ = 
\frac{1}{\left(j!\right)^{\frac{1}{j}}} 
\left(A \pi^{\frac{1}{\mu j}} + \frac{1}{2\pi} \left( 2 \pi\right)^{\frac{1}{\mu j}} \left( \Gamma(\mu j +1) \right)^{\frac{1}{\mu j}}
\right)^{\beta} M^{\beta} \\
&\ \ \ \ \leq \frac{1}{\left(j!\right)^{\frac{1}{j}}} 
\left(A^{\beta} \pi^{\frac{\beta}{\mu j}} + \left(\frac{1}{2\pi}\right)^{\beta} \left( 2 \pi \right)^{\frac{\beta}{\mu j}} 
\left(\Gamma(\mu j +1)  \right)^{\frac{\beta}{\mu j}}
\right) M^{\beta} , 
\end{align*}
We have $\frac{1}{\left(j!\right)^{\frac{1}{j}}} A^{\beta} \pi^{\frac{\beta}{\mu j}} \to 0$, therefore it remains
to examine the limit of
\begin{equation}\label{exp6E}
 \frac{1}{\left(j!\right)^{\frac{1}{j}}} \left(\frac{1}{2\pi}\right)^{\beta} \left( 2 \pi \right)^{\frac{\beta}{\mu j}} 
\left( \Gamma(\mu j +1) \right)^{\frac{\beta}{\mu j}} M^{\beta}  .
\end{equation}
From \eqref{exp6B}  we have  $j! \sim \sqrt{2\pi j} \left(\frac{j}{e}\right)^{j} $ therefore
$ \left( j! \right)^{\frac{1}{j}} \sim \left(2\pi j\right)^{\frac{1}{2j}} \frac{j}{e} \sim  \frac{j}{e} .$
From \eqref{exp6C}  we have  $\Gamma(\mu j +1)  \sim \sqrt{2\pi \mu j} \left(\frac{\mu j}{e}\right)^{\mu j} $ therefore
\[ \left( \Gamma(\mu j +1)  \right)^{\frac{\beta }{\mu j}} \sim \left(2\pi \mu j\right)^{\frac{\beta}{2\mu j}} 
\left(\frac{\mu j}{e}\right)^{\beta} \sim 
 \left(\frac{\mu j}{e}\right)^{\beta} .\]
Therefore
\[
\eqref{exp6E}\  \sim \left(\frac{M}{2\pi}\right)^{\beta} \frac{\left(\frac{\mu j}{e}\right)^{\beta}}{  \frac{j}{e}}
\to \begin{cases} 
0 &,\ \ \text{if}\ \ \beta < 1 \\
  \frac{M \mu}{2\pi}  &,\ \ \text{if}\ \ \beta =1
\end{cases}  .
\]
The limit is smaller than 1 in all cases, therefore the series converges.

Summing up, we obtain 
\[
T[S](t,x)\leq C\left(1+ \exp\left\{\norm{S(t,\cdot)}{L^{\infty}}^{\beta}\right\}\right) \leq C + C \norm{\rho(t,\cdot)}{L^p}^{\beta}\ .
\]
Recall that we have choosen $p<2$
such that the Lemma \ref{lem:dispersion} applies. 
We end up with
\begin{equation*}
\norm{\rho(t,x)}{L^{p}}\leq t^{-\lambda}\|f_0(x,v)\|_{L^{1}_{x}L^{p}_{v}}   
 + C 
\int_{0}^{t}\Big(1+
\norm{\rho(s,x)}{L^{p}_{x}}^{\beta}\Big) \frac{ds}{(t-s)^\lambda}\ ,
\end{equation*}
where $\lambda = 2/p' <1$ so that we can bootstrap. 
\end{proof}

\section{(Almost) $L^\infty$ growth in dimension $3$}\label{d=3}

\subsection{Almost $L^\infty$ growth}

In this Section we consider the kinetic model \eqref{kinmodel} in $d=3$ dimensions under hypothesis \eqref{lra2}.

\begin{proof}[Proof of Theorem \ref{lra1}.]
If $1\leq r<3$, $\alpha >0$ and $T[S]$ satisfies \eqref{lra2} then $T[S]$ can be estimated {\em a priori} in terms of the mass $M$.
Indeed,  
\[\norm{S(t,\cdot)}{L^{r}(\R^3)}\leq \norm{G}{L^{r}(\R^3)}\norm{\rho(t,\cdot)}{L^1(\R^3)}\leq C M ,\]
because $G(x) \sim \frac{C}{|x|}$ for small $|x|$, and $G(x)$ decays exponentially for large $x$.
Therefore $T[S](t,x,v,v') \leq C + C M^{\alpha} $ and global existence follows easily.

Assume now that $3\leq r <\infty$ and $0<\alpha < \frac{r}{r-3}$.
Choose $p$ defined by \[ \frac1{p'} = \frac{\alpha(r-3)}{3r}<\frac13\ , \] and define $B$ such that
\[
\frac1{B'} =  \frac13 - \frac1r = \frac1{\alpha p'}\  .
\]


Using fractional integration \cite{LiebLoss} we get for the signal $S = G*\rho\leq \frac C{|x|}*\rho(x)$ (both short and long range parts),
\begin{align*}
\norm{S}{L^r}&\leq C \norm{\frac{1}{|x|} \ast \rho}{L^r}
\leq C
 \norm{\rho}{L^{B}}
 \leq C M^{1-\frac{p'}{B'}}\norm{\rho}{L^p}^{\frac{p'}{B'}} .
\end{align*}
Consequently we get the crucial estimate required in Lemma \ref{lem:dispersion}:
\[T[S](t,x)\leq C + C  \norm{S}{L^r}^\alpha \leq C+C \norm{\rho}{L^p} \ , \]
where $p$ is smaller than $3/2$.
%
We can complete the proof as in  Theorem \ref{exp3C}.

\end{proof}

\subsection{$L^\infty$ growth: global existence for small data}

\begin{proof}[Proof of Theorem \ref{3dsmall}]
We have 
\begin{equation}
 \label{sd1}
 \partial_{t} f + v \cdot\nabla_{x} f \leq   C \int_{V} \Big(1+\norm{S(t)}{L^\infty}\Big) f(t,x,v')\ dv' = C \Big(1+ \norm{S(t)}{L^\infty}\Big) \rho(t,x) .
\end{equation}
To apply the Strichartz estimate \cite{CP} we need four parameters $q,p,r,a$ such that
\begin{subequations}\label{sd30}
\begin{align}
 &1\leq r \leq p \leq \infty \label{sd3a}\\
 & 0 \leq \frac{1}{r} - \frac{1}{p} < \frac{1}{3} \\
 & 1 \leq \frac{1}{r} + \frac{1}{p}  \label{sd3c}\\
 &\frac{2}{q}= 3 \left( \frac{1}{r} - \frac{1}{p}\right) \label{sd3d}\\
 &a=\frac{2pr}{p+r} \label{sd3e}
\end{align}
\end{subequations}
More conditions will be imposed later. We get:
\begin{align}
 \norm{f}{L^{q}_{t} L^{p}_{x} L^{r}_{v}} & \leq \norm{f_0}{L^{a}_{x,v}} + C
 \Big\| (1+ \norm{S(t)}{L^\infty}) \rho(t,x)\Big\|_{L^{q'}_{t} L^{r}_{x} L^{p}_{v}} \nonumber \\
 & = \norm{f_0}{L^{a}_{x,v}} + C(|V|)
 \norm{ (1+ \norm{S(t)}{L^\infty}) \norm{\rho(t,x)}{L^{r}_{x}} \  }{L^{q'}_{t}  }  \label{sd2c}.
\end{align}
In the sequel we omit the constant part in the growth of the turning kernel for the sake of clarity. Assume 
\begin{equation}
 \label{sd7}
 p > \frac{3}{2} .
\end{equation}
Then $p'<3$ therefore,
\begin{equation}\label{sd21}
 \norm{S(t)}{L^\infty}\leq \norm{G * \rho(t)}{L^\infty} \leq \norm{G}{L^{p'}} \norm{\rho(t)}{L^{p}} 
 \leq C \norm{\rho(t)}{L^{p}} ,
\end{equation}
because $G(x) \sim \frac{C}{|x|}$ for small $|x|$, and $G(x)$ decays rapidly for large $|x|$.
Moreover, since $r\leq p$ we have by interpolation,
\begin{equation*}\label{sd22}
 \norm{\rho(t)}{L^r} \leq \norm{\rho(t)}{L^1}^{1-\frac{p'}{r'}} \norm{\rho(t)}{L^p}^{\frac{p'}{r'}} 
 = M^{1-\frac{p'}{r'}} \norm{\rho(t)}{L^p}^{\frac{p'}{r'}} .
\end{equation*}
Therefore
\begin{align*} 
 \norm{\ \norm{S(t)}{L^\infty} \norm{\rho(t,x)}{L^{r}_{x}}\  }{L^{q'}_{t} } 
&  \leq C  \norm{\ \norm{\rho(t)}{L^p} \ \norm{\rho(t)}{L^{p}}^{\frac{p'}{r'}} \ }{L^{q'}_{t}}\\
&=  \norm{\ \norm{\rho(t)}{L^p}   \ }{L^{q'\left(1 + \frac{p'}{r'} \right)}_{t}}^{ 1 + \frac{p'}{r'}  }
\end{align*}
Now
\[
\norm{\rho(t)}{L^p}=\norm{f(t,x,v)}{L^{p}_{x} L^{1}_{v}} \leq C(|V|) 
\norm{f(t,x,v)}{L^{p}_{x} L^{r}_{v}}
\]
therefore
\begin{equation*}
\norm{\ \norm{S(t)}{L^\infty} \norm{\rho(t,x)}{L^{r}_{x}}\  }{L^{q'}_{t} } 
\leq C
 \norm{f(t,x,v)}{L^{ q'\left(1 + \frac{p'}{r'} \right) }_{t} L^{p}_{x} L^{r}_{v}}^{ 1 + \frac{p'}{r'}  } 
\label{sd23e}.
\end{equation*}
Suppose that
\begin{equation}
 \label{sd24}
 q'\left(1 + \frac{p'}{r'} \right) = q .
\end{equation}
Then 
\begin{equation}
 \label{sd25}
 \norm{\ \norm{S(t)}{L^\infty} \norm{\rho(t,x)}{L^{r}_{x}}\  }{L^{q'}_{t} } 
 \leq C \norm{f(t,x,v)}{L^{ q }_{t} L^{p}_{x} L^{r}_{v}}^{ 1 + \frac{p'}{r'}  } 
\end{equation}
and plugging this into \eqref{sd2c} we get
\begin{equation*}
  \norm{f(t,x,v)}{L^{ q }_{t} L^{p}_{x} L^{r}_{v} } \leq \norm{f_0}{L^{a}_{x,v}} + 
  C \norm{f(t,x,v)}{L^{ q }_{t} L^{p}_{x} L^{r}_{v}}^{ 1 + \frac{p'}{r'}  } 
\end{equation*}
If $\norm{f_0}{L^{a}_{x,v}}$ is small enough then we can bootstrap.

\medskip

We need to verify that there exist $(q,p,r,a)$ satisfying \eqref{sd30}, \eqref{sd7} and \eqref{sd24}. There are many possible
choices. For example, if we want initial data $f_0 \in L^{a}_{x,v}$ with
$a=\frac{3}{2}$ (critical exponent in dimension 3) we must choose $p$ and $r$ so that 
$\frac{1}{p}+\frac{1}{r}=\frac{4}{3}$. 
The complete set of exponents solving these constraints is:
\[q=1+\sqrt{2}\ ,\ p=\frac{9+3\sqrt{2}}{7}\ ,\ r= 3 \left(\sqrt{2} -1 \right) \ ,\]
where all conditions are fulfilled.

\end{proof}

\subsection{Sublinear $L^\infty$ growth}\label{sublinear}

To close this section we
give a quick sketch of the observation in \cite{CMPS} that the hypothesis
\[
 0\leq T[S](t,x,v,v') \leq C\Big( 1 + \norm{S(t,\cdot)}{L^{\infty}}^{\alpha}\Big)\ ,
\]
implies global existence.
Fix $p$ and $q$ such that
\[ \frac{\alpha}3 + \frac1p = \frac1q \ , \quad p>\frac32\ . \]
Then we have the following elliptic estimate (see below),
\begin{equation}\label{ell}
\norm{S(t,\cdot)}{L^\infty}= \norm{G * \rho(t)}{L^\infty}
\leq C(M) \norm{\rho(t)}{L^p}^{p'/3}\ .
\end{equation}
Therefore (again omitting the constant contribution of the turning kernel)
\begin{equation*}
f(t,x,v)\leq 
f_0(x-tv,v) + C \int_{0}^{t}\norm{\rho(s)}{L^p }^{\alpha}   \rho(s,x-(t-s)v)  ds .
\end{equation*}
Take the $L^p_xL^q_v$ norm and use the dispersion estimate with $\lambda = 3(1/q -1/p) = \alpha$ to get
\begin{align}
\norm{f(t)}{L^{p}_xL^q_v}
&\leq t^{-\alpha}\|f_0(x,v)\|_{L^{q}_{x}L^{p}_{v}} + |V|^{1/p} \int_{0}^{t} \frac1{(t-s)^\alpha} \norm{\rho(s)}{L^p }^{p' \alpha/3} 
\norm{ \rho(s) }{L^{q}_{x}}  ds  \\
& \leq t^{-\alpha}\|f_0(x,v)\|_{L^{q}_{x}L^{p}_{v}} + C \int_{0}^{t} \frac1{(t-s)^\alpha} \norm{\rho(s)}{L^p }^{p' \alpha/3 + p'/q'}\ ds , \label{lto.5} 
\end{align}
where $p' \alpha /3 + p'/q' = 1$ by definition.

To prove the elliptic estimate \eqref{ell} write
\[S\leq C \rho * \frac{\chi_{|x|\leq R}}{|x|} + C \rho * \frac{\chi_{|x|\geq R}}{|x|}\]
Then, if $p'<3$, 
\begin{align*}
\|S\|_{L^\infty}& \leq C \|\rho\|_{L^p} \|\frac{\chi_{|x|\leq R}}{|x|}\|_{L^{p'}} + C 
\|\rho\|_{L^1} \|\frac{\chi_{|x|\geq R}}{|x|}\|_{L^{\infty}}\\
&\leq C \Big( \|\rho\|_{L^p} R^{\frac{3}{p'}-1} + \|\rho\|_{L^1} R^{-1} \Big) \ . 
\end{align*}
Choose $R$ so that $\|\rho\|_{L^p} R^{\frac{3}{p'}-1}=\|\rho\|_{L^1} R^{-1} $, {\em i.e.} choose
\[R=\left(\frac{\|\rho\|_{L^1}}{\|\rho\|_{L^p}}\right)^{p'/3} .\]
This gives 
\[\|S\|_{L^\infty} \leq C \|\rho\|_{L^p}^{p'/3} \|\rho\|_{L^{1}}^{1-p'/3} = C M^{1-p'/3} \|\rho\|_{L^p}^{p'/3}\ . \]

\section{Extension to internal dynamics}

\label{sec:internal}

Recall the kinetic model with internal dynamics:
\begin{subequations}\label{kinmodel_y}
\begin{align}
\partial_{t} p + v\cdot\nabla_{x} p + \nabla_y \cdot \Big( G(y,S) p \Big)= &  \int_{ v'\in V} \lambda[y] K(v,v') p(t,x,v',y) dv'
 \nonumber\\
   & \qquad \qquad  -  \lambda[y]  p(t,x,v,y)   \  ,\\
\label{eq:mean field} -\Delta S + S= \rho\  , &  
\end{align}
\end{subequations}
Assuming that $K$ is bounded we reduce to $K = 1/|V|$ without loss of generality.

This model takes into account the transport along characteristics of the internal cellular dynamics
\[ \frac{dy}{dt} = G(y,S(t,x))\ , \quad y\in \R^m\ . \]
For {\em E. coli}, the regulatory network described by $G$ is made of six main proteins essentially (named Che-proteins), and the main events are methylation and phosphorylation. 
Indeed, in the absence of a chemoattractant (basal activity), the phosphorylated protein Che-Y is supposed to diffuse inside the cell and to reach the flagella motor complex, enhancing switch between CCW rotation and CW rotation, that is tumbling. This transduction pathway is in fact inhibited when the chemoattractant (say aspartate) binds a membrane receptor, triggering methylation of the membrane receptor complex, and eventually inhibition of the tumbling process.

This network exhibits a remarkable excitation/adaptation behavior, which is crucial for cell migration. For the sake of simplicity, one deals in general with a system of two coupled ODEs which captures the same features. This system should be {\em excitable} with slow adaptation -- there is a single, stable equilibrium state, but a perturbation above a small threshold triggers a large excursion in the phase plane (see figure \ref{fig:FHN} ({\em left})) -- and possibly one-sided -- in case of positive chemotaxis, the cells do not respond specifically to a decrease of the chemoattractant concentration \cite{BB74}. This characterization of dynamical systems is very well known in biological modeling, as it is the basis of the FitzHugh-Nagumo models \cite{Murray} for potential activity in axons. Furthermore, it is often associated to the phenomenon of pulse wave propagation ({\em e.g.} calcium waves), \cite{Keener}. In the context of cell migration, it is also involved in the slime mold amoebae {\em D. discoideum} aggregation process, where the chemoattractant cAMP is relayed by the cells \cite{Hofer,DolakSchmeiser}.

To be more concrete, the following set of equations is generally proposed \cite{ErbanOthmer04}
\begin{equation}
\left\{ \begin{array}{rll}
  \dfrac {dy_1}{dt} = & \dfrac1{\tau_e} \big( h(S) - (y_1+y_2) \big) \quad & {\mbox (excitation)} \vspace{.2cm}\ , \\
  \dfrac {dy_2}{dt}  = & \dfrac1{\tau_a} \big( h(S) - y_2 \big) \quad & {\mbox (adaptation)} \ .
 \end{array}
\right.
\end{equation}
Considered to be decoupled from the transport equation, these two internal quantities relax respectively to 
\[ \lim_{t\to \infty} y_1 = 0 \ , \quad \lim_{t\to \infty} y_2 = h(S)\ , \] with a slow time scale associated to adaptation provided that $\tau_e\ll \tau_a$. However, this system cannot reproduce true excitability with a large gain factor for small perturbations because it is linear with respect to the variable $y$.  
In a slightly different context (pulsatory cAMP waves), Dolak and Schmeiser considered an even simpler system \cite{DolakSchmeiser}, namely
\begin{equation}
\left\{ \begin{array}{rll}
  y_1 = &\big( h(S) - y_2 \big)_+ \quad & {\mbox (excitation)} \ , \vspace{.2cm} \\
  \dfrac {d y_2}{dt}  = &\dfrac1{\tau_a} \big( h(S) - y_2 \big) \quad & {\mbox (adaptation)} \ .
 \end{array}
\right.
\end{equation}
This particular choice does select responses to one-sided stimuli, but fails for true excitability. We suggest to consider the following phenomenological translated slow-fast, FHN type, system,
\begin{equation}
\left\{ \begin{array}{rll}
  \dfrac {dy_1}{dt} = & \dfrac1{\tau_e} \big(h(S) - q(y_1) - y_2   \big) \quad & {\mbox (excitation)} \vspace{.2cm}\ , \\
  \dfrac {dy_2}{dt}  = & \dfrac1{\tau_a} \big( h(S) + y(1) - y_2 \big) \quad & {\mbox (adaptation)} \ ,
 \end{array}
\right.
\end{equation}
where $q$ is a cubic function depicted in figure \ref{fig:FHN}.

\begin{figure}
\includegraphics[width = .48\linewidth]{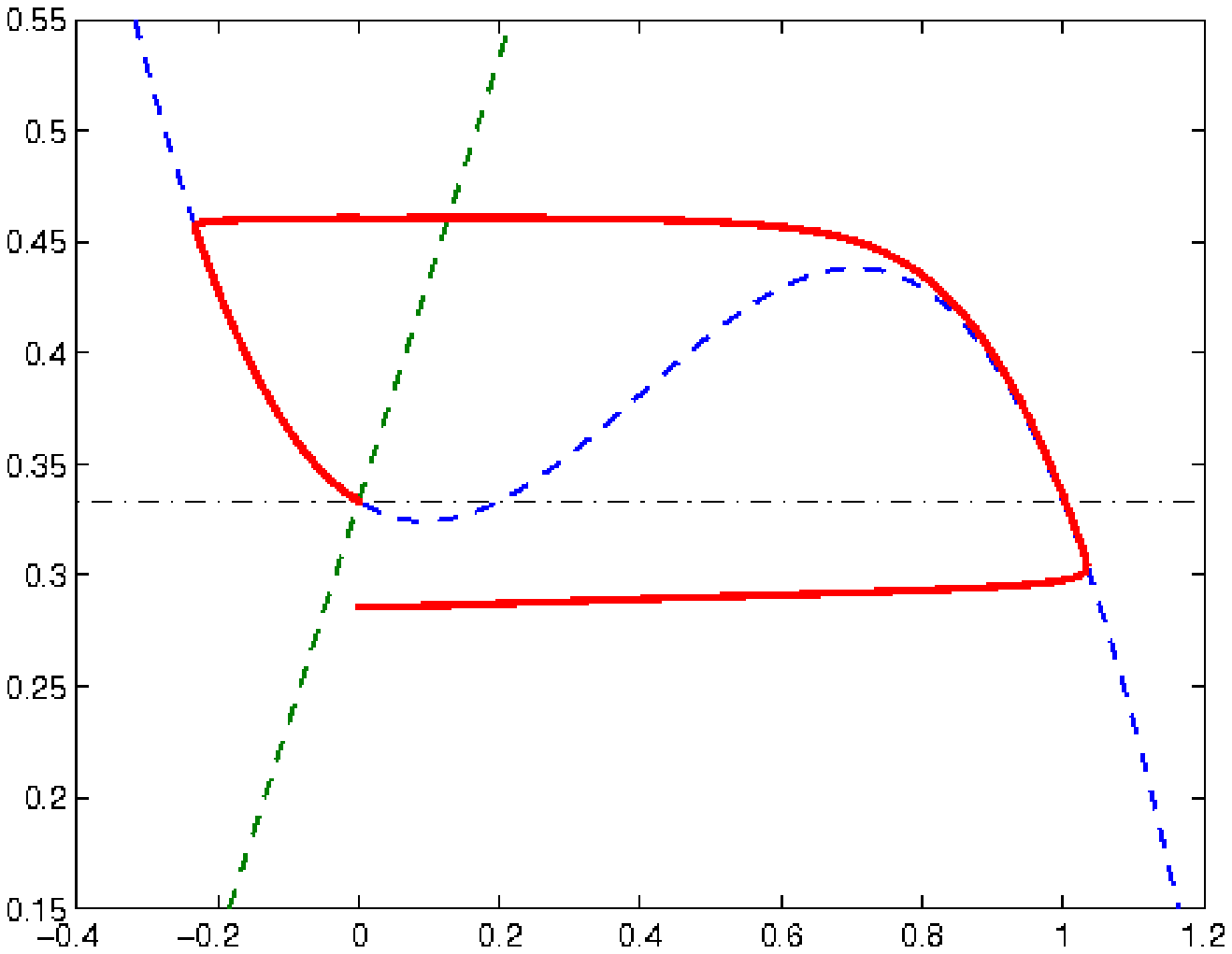} \,
\includegraphics[width = .48\linewidth]{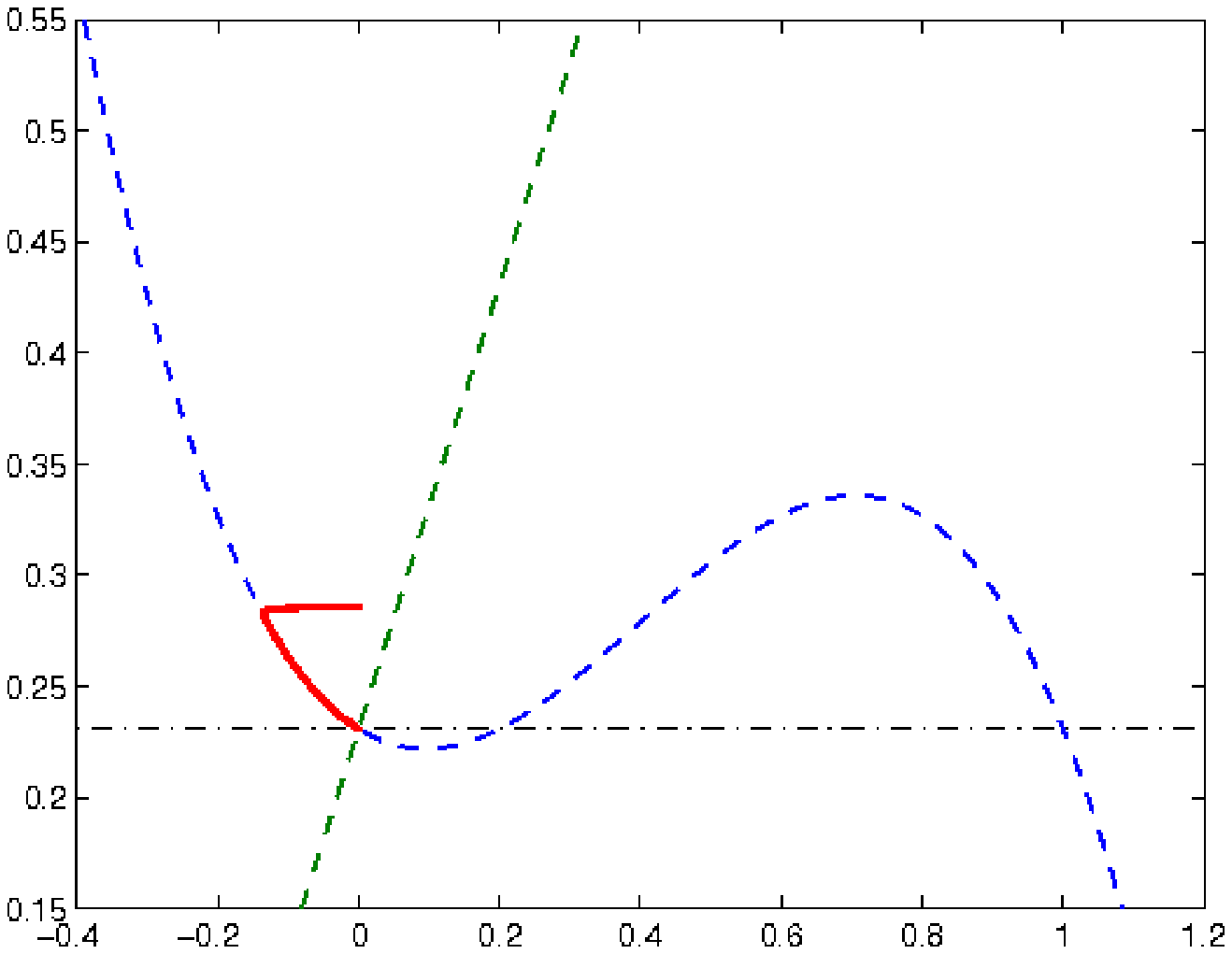} 
 \caption{A two coupled ODE system exhibiting short time excitation/mid time adaptation with one-sided selection. The picture is the same as for the FHN model. The perturbation of the equilibrium state is enhanced by a displacement of the basal line $y2 = h(S0+dS)$, translating the system up ({\em left}) or down ({\em right}). Here the cubic function is given by $q(u) = u(u-1)(u-.2)$, the saturating ligand function is given by $h(S) = S/(1+S)$, the basal aspartate concentration is $S0 = .4$ and the system reacts to  perturbations $dS =  .1$.} 
\label{fig:FHN}
\end{figure}

\begin{proof}[Proof of Theorem \ref{the:internal}.]
Our next step is to prove global existence under general and manageable assumptions settled in Theorem \ref{the:internal}. But let us begin with an important remark on the methodology.

\begin{remark}
To obtain {\em a priori} estimates, one possible strategy would be to use the characteristics to handle with \eqref{kinmodel_y}, as it is performed in \cite{ErbanHwang} in 1D. For this purpose, integrating the hyperbolic \eqref{kinmodel_y}
along the backward-in-time auxiliary problem
\[ \dot{X}(s) = v \ , \quad \dot{Y}(s) = G(Y,S(s,X)) \ ,\quad (X(t),Y(t)) = (x,y)\ , \] gives the estimate
\[  \dfrac{d}{ds} f(s,X(s),v,Y(s)) - \Big(\nabla_y\cdot G\Big) p  \leq \lambda[Y] \mu(s,X,Y)\ . 
\]
The difficulty arises at two levels here. First one has to control the $\nabla_y\cdot G$ contribution, and secondly one has to perform later on the change of variables $z = Y(y)$. This induces a Jacobian contribution $\left| \dfrac{\partial Y}{\partial y} \right|^{-1}$, and one has to control it too. In the sequel we avoid these two difficulties by working on averaged quantities.
\end{remark}

We use a partial representation formula of the solution from the free transport operator $\partial_t p + v\cdot \nabla_x p$. First integrating the equation with respect to $y$, we obtain
\[ \partial_t f + v\cdot \nabla_x f + 0 \leq \frac1{|V|} \int_y \lambda[y] \mu(t,x,y)\ dy\ , \] 
so that  
\[ f(t,x,v) \leq f_0(x-tv,v) + \frac1{|V|} \int_{s=0}^t \int_y \lambda[y] \mu(s,x-(t-s)v,y)\ dy ds\ . \]
Using the $L^p_xL^1_v$ dispersion Lemma \ref{lem:dispersion} we get as usual
\begin{align}
 \|\rho(t)\|_{L^p} &\leq  \|f_0(x-tv,v)\|_{L^{p}_{x}L^{1}_{v}}  + \frac1{|V|} \int_{s=0}^t \left\| \int_y \lambda[y] \mu(s,x-(t-s)v,y)\ dy \right \|_{L^p_xL^1_v}  ds \nonumber\\
&\leq t^{-\lambda}\|f_0(x,v)\|_{L^{1}_{x}L^{p}_{v}}   \nonumber
\\ &\qquad +  |V|^{1/p-1} \int_{s=0}^t \frac1{(t-s)^\lambda} \iint_{x,y} \lambda[y] \mu(s,x,y)\ dx dy ds\ ,\label{eq:estimate rho(y)}
\end{align}
where $\lambda=3/p'$.
We now use the two growth assumptions on $\lambda$ and $G$: \[\lambda[y]\leq C(1+|y|)\ , \quad |G|(y,S)\leq C(1 +  |y| + S^\alpha) \ ,\   0\leq \alpha<1\ , \] to control the time growth of the average quantity $\iint_{x,y} |y| \mu(t-s,x,y)\ dx dy $.

\begin{remark} Note that in dimension $d=2$ we can handle  any nonnegative $\alpha$. \end{remark}

We test the master equation \eqref{kinmodel_y} against $|y|$:
\[
 \dfrac{d}{d s} \iint_{x,y} |y| \mu(s,x,y) \ dxdy + 0 + \iint_{x,y} |y| \nabla_y\cdot (G(y,S) \mu(s,x,y))\ dy dx = 0\ ,
\]
therefore, using $|G|\leq C(1+|y|+S^\alpha)$,
\begin{align}  
\dfrac{d }{ds} \iint_{x,y} |y| \mu(s) \ dxdy & = \iint_{x,y} \frac{y}{|y|}\cdot G(y,S) \mu(s,x,y)\ dy dx\nonumber \\
& \leq  \iint_{x,y}  |G|(y,S) \mu(s,x,y)\ dy dx\nonumber \\
& \leq C + C \iint_{x,y} |y| \mu(s,x,y)\ dy dx + C \int_x |S(s,x)|^\alpha \rho(s,x)\ dx \ . \label{eq:preDuhamel}
\end{align}
  
  \begin{remark}
If we agree to diminish $\alpha$ it is possible to deal with higher exponents in $\lambda[y]\leq C(1+|y|^\gamma)$. For instance we have by Young's inequality:
\begin{align*} \dfrac{d }{ds} \iint_{x,y} |y|^\gamma \mu(s) \ dxdy & \leq \gamma \iint_{x,y} |y|^{\gamma-1} |G|(y,S) \mu(s,x,y)\ dy dx \\ &  \leq  C \iint_{x,y}  |y|^\gamma  \mu(s,x,y)\ dy dx \\ & \quad +  C\gamma \iint_{x,y} \left(\frac{\gamma - 1}\gamma |y|^\gamma + \frac 1\gamma |S(s,x)|^{\alpha\gamma}\right) \mu(s,x,y)\ dy dx \ .
  \end{align*}
The same argument follows provided that $\alpha\gamma<1$. More general combination of exponents could have been considered. We have chosen here a simple framework for the sake of clarity.
\end{remark}

We can use the Duhamel formula to represent the  inequality \eqref{eq:preDuhamel} as
\begin{multline*}
\iint_{x,y} |y| \mu(s,x,y) \ dxdy \leq Ce^{Cs} + e^{Cs} \iint_{x,y} |y| \mu_0(x,y)\ dx dy \\
+ C \int_{\tau = 0}^s e^{C(s-\tau)}  \int_x |S(\tau,x)|^\alpha \rho(\tau,x)\ dx d\tau\ .
\end{multline*}

Plugging that into \eqref{eq:estimate rho(y)} gives
\begin{equation*}
\|\rho\|_{L^p} \leq C_0 (t) + C \int_{s=0}^t \frac1{(t-s)^\lambda} \int_{\tau = 0}^s e^{C(s-\tau)}   \int_{x}  |S(\tau,x)|^\alpha \rho(\tau,x)\ dx d\tau ds \ .
\end{equation*}
We choose $p<3/2$ so that $\lambda = 3/p'<1$. Since $\alpha < 1$ we have $3< \frac{3}{\alpha}$ and we can choose $p$ sufficiently close to $3/2$ so that $ 3 < p' < \frac{3}{\alpha}$.
Then 
\[
 \int S(t,x)^\alpha \rho(t,x) dx \leq \norm{S^\alpha}{L^{p'}} \norm{\rho}{L^{p}} = \norm{S}{L^{\alpha p'}}^{\alpha} \norm{\rho}{L^p} \ .
\]
>From the mean field chemical equation \eqref{eq:mean field} $-\Delta S+S=\rho$ we deduce the following elliptic estimate. 
We have $\alpha p' < 3$, and $S=G*\rho$ 
where $G(x) \sim \frac{C}{|x|}$ for small $|x|$ (short range) and $G(x)$ decreases exponentially fast for large $|x|$ (long range). Thus we obtain
\[
 \norm{S}{L^{\alpha p'}} = \norm{G * \rho}{L^{\alpha p'}} \leq \norm{G}{L^{\alpha p'}} \norm{\rho}{L^1} \leq C M\ ,
\]
therefore
\[
 \int S(t,x)^\alpha \rho(t,x) dx \leq C M^\alpha \norm{\rho}{L^p}\ .
\]
We obtain
\begin{align*}
 \|\rho(t)\|_{L^p} &\leq C_0 (t) + C\int_{0}^{t} \frac1{(t-s)^\lambda} \int_{\tau = 0}^s e^{C(s-\tau)} \norm{\rho(\tau)}{L^p} d\tau ds \\
 & \leq C_{0}(t) + C \int_{0}^{t}  \norm{\rho(\tau)}{L^p} \int_{\tau}^{t} \frac{e^{C(s-\tau)}}{(t-s)^\lambda} ds d\tau\ , 
\end{align*}
Using the boundedness of $\int_{s = \tau}^t \frac1{(t-s)^\lambda}  e^{C(s-\tau)}\ ds$ with respect to $\tau$, we conclude thanks to a Gronwall estimate.

\end{proof}

\end{document}